\documentclass[10pt]{amsart}

\usepackage{amssymb,amscd}

\begin{document}

\newtheorem{thm}{Theorem}[section]
\newtheorem{lem}[thm]{Lemma}
\newtheorem{prop}[thm]{Proposition}
\newtheorem{cor}[thm]{Corollary}
\newtheorem{exmp}[thm]{Example}
\newtheorem{rem}[thm]{Remark}
\newtheorem{ques}[thm]{Question}
\newtheorem{defn}[thm]{Definition}

\input amssym.def

\newcommand{\ilim}{\mathop{\varinjlim}\limits}

\def\Jac{\mbox{\rm Jac$\:$}}
\def\Nil{\mbox{\rm Nil$\:$}}
\def\wt{\widetilde}
\def\Q{\mathcal{Q}}
\def\O{\mathcal{O}}
\def\ff{\frak}
\def\Spec{\mbox{\rm Spec}}
\def\ZS{\mbox{\rm ZS}}
\def\wB{\mbox{\rm wB}}
\def\type{\mbox{ type}}
\def\Hom{\mbox{ Hom}}
\def\rank{\mbox{ rank}}
\def\Ext{\mbox{ Ext}}
\def\Ker{\mbox{ Ker }}
\def\Max{\mbox{\rm Max}}
\def\End{\mbox{\rm End}}
\def\ord{\mbox{\rm ord}}
\def\l{\langle\:}
\def\r{\:\rangle}
\def\Rad{\mbox{\rm Rad}}
\def\Zar{\mbox{\rm Zar}}
\def\Supp{\mbox{\rm Supp}}
\def\Rep{\mbox{\rm Rep}}
\def\cal{\mathcal}
\def\p{{\rm{p}}}

\def\gen{\mbox{\rm gen$\:$}}
\def\Jac{\mbox{\rm Jac$\:$}}
\def\Nil{\mbox{\rm Nil$\:$}}
\def\ord{\mbox{\rm ord}}
\def\wt{\widetilde}
\def\Q{\mathcal{Q}}
\def\ff{\frak}
\def\Spec{\mbox{\rm Spec}}
\def\ZS{\mbox{\rm ZS}}
\def\wB{\mbox{\rm wB}}
\def\type{\mbox{ type}}
\def\Hom{{\rm Hom}}
\def\depth{{\rm depth}}
\def\Gen{{\rm Gen}}
\def\rank{\mbox{\rm rank}}
\def\Ext{{\rm Ext}}
\def\embdim{{\mbox{\rm emb.dim } }}
\def\length{{\mbox{\rm length }}}
\def\Ker{\mbox{\rm Ker }}
\def\pd{{\rm pd}}
\def\gr{{\rm gr}}
\def\Max{\mbox{\rm Max}}
\def\End{\mbox{\rm End}}
\def\l{\langle\:}
\def\r{\:\rangle}
\def\Rad{\mbox{\rm Rad}}
\def\Zar{\mbox{\rm Zar}}
\def\Supp{\mbox{\rm Supp}}
\def\Rep{\mbox{\rm Rep}}
\def\cal{\mathcal}
\def\p{{\rm{p}}}
 


\title{Generic formal fibers and analytically ramified stable rings}

\author{Bruce Olberding}


\address{Department of Mathematical Sciences, New Mexico State University,
Las Cruces, NM 88003-8001}

\maketitle


\begin{abstract} Let $A$ be a local Noetherian domain of Krull dimension $d$.  Heinzer, Rotthaus and Sally have shown that if the generic formal fiber of $A$ has dimension $d-1$, then $A$ is 
birationally dominated by a one-dimensional analytically ramified local Noetherian ring having residue field finite over the residue field of $A$.  We explore further this correspondence between prime ideals in the generic formal fiber and one-dimensional analytically ramified local rings.  Our main focus is on the case where the analytically ramified local rings are stable, and we show that in this case the embedding dimension of the stable ring reflects the embedding dimension of a prime ideal maximal in the generic formal fiber,  thus providing a measure of how far the generic formal fiber deviates from regularity.  A number of  characterizations of  analytically ramified local stable domains are also given.        
\end{abstract}


\section{Introduction}

An important technical  fact regarding a finitely generated algebra $A$ over a field or the ring of integers, is that if $A$ is reduced, then $A$ has {\it finite normalization}, meaning that the integral closure of $A$ in its total ring of quotients is a finitely generated $A$-module.  Geometrically, this guarantees that the normalization of  a variety or arithmetic scheme is given by a finite morphism.  However, outside the geometric and arithmetic contexts, a local Noetherian domain, even of dimension 1, need not have finite normalization.  Well-known examples due to Akizuki, Schmidt and Nagata attest to this; see \cite{Akizuki}, \cite{Schmidt} and \cite[Example 3, p.~205]{Na}.  These examples are constructed between a carefully chosen rank one discrete valuation ring and its completion, and as such, the examples suggest a certain esoteric quality regarding the failure of a one-dimensional local Noetherian domain to have finite normalization.  However, an entirely different construction due to Heinzer, Rotthaus and Sally suggests a certain inevitability to such examples, and shows that local Noetherian domains without finite normalization can be found birationally dominating any $d$-dimensional local Noetherian domain, $d>1$, having generic formal fiber of dimension $d-1$ (which is the case if $A$ is essentially of finite type over a field) \cite[Corollary 1.27]{HRS}.

More precisely, let  $A$ be a local
Noetherian domain with maximal ideal ${\ff m}$ and  quotient field $F$, and let $\widehat{A}$ denote the completion of $A$ in the ${\ff m}$-adic topology.    The {\it generic
formal fiber} of $A$ is the localization of $\widehat{A}$ at the multiplicatively closed subset consisting of the nonzero elements of $A$.  Thus the generic formal fiber of $A$ is the   
  Noetherian ring $\widehat{A}[F]$, and the prime ideals of $\widehat{A}[F]$ are those prime ideals  extended from prime ideals $P$ of $\widehat{A}$ such that $P \cap A = 0$.    
   When  $A$ has Krull dimension $d>0$, then the generic formal fiber
of $A$ has Krull dimension  less than $d$.  Heinzer, Rotthaus and Sally have shown that if the generic formal fiber is as large as possible in the sense that its Krull dimension is $d-1$, then there exists an analytically ramified local Noetherian domain $R$ birationally dominating $A$ and having residue field finite over the residue field of $A$ \cite[Corollary 1.27]{HRS}.  A theorem of Krull asserts that the property of being {\it analytically ramified}, meaning that the completion contains nonzero nilpotent elements, is equivalent for  one-dimensional local Noetherian domains to the failure to have finite normalization \cite{Krull}.  
The description  of the rings in \cite{HRS} is quite straightforward: Given a prime ideal $P$ of $\widehat{A}$ of dimension $1$ such that $P \cap A =0$, the ring $F \cap (\widehat{A}/I)$ is a one-dimensional analytically ramified local Noetherian ring for appropriate choices of $P$-primary ideal $I$ (this is restated more formally in Lemma~\ref{HRS lemma} below).      
Moreover, Matsumura has shown such a prime ideal $P$ exists in $\widehat{A}$ whenever $A$ 
 is essentially of finite type over a field, so the construction is relevant in many natural  
circumstances \cite[Theorem 2]{Mat2}.

In this article we explore further the connection between the generic formal fiber of $A$ and the nature of the analytically ramified one-dimensional local rings which birationally dominate it. Specifically we show the embedding dimension of birationally dominating analytically ramified ``stable'' local rings reflects the regularity, or lack thereof, of the generic formal fiber.  To phrase this more precisely, we introduce some terminology.   
Recall that if  $A \subseteq R$ is an extension of
quasilocal domains, then $R$
{\it dominates} $A$ if the maximal ideal of $A$ is a subset of the maximal ideal of $R$, and if $A$ and $R$ share the same quotient field, then 
 $R$
 {\it birationally dominates} $A$.  When $R$ birationally dominates $A$ and $R/{\ff m}R$ is a finite $A$-module, then we say that $R$ {\it finitely dominates} $A$; if also $R = A + {\ff m}R$, then $R$ {\it tightly dominates} $A$.  Thus in our terminology, the theorem of Heinzer, Rotthaus and Sally states that  when  the generic formal fiber of the $d$-dimensional local Noetherian domain $A$ has dimension $d-1$, then $A$ is finitely dominated by  a one-dimensional analytically ramified local Noetherian domain. 
 
 Our focus is on a specific class of one-dimensional analytically ramified local Noetherian domains, those that are ``stable'' in the sense of Lipman \cite{Lipman} and Sally and Vasconcelos \cite{SV}. These rings, which we define in Section 2, are characterized in the local Noetherian case by the property that each ideal has a principal reduction of  reduction number at most $1$; that is, for each ideal $I$ of $R$, there exists $i \in I$ such that $I^2 = iI$.   As we discuss in Section 2, analytically ramified local Noetherian stable domains of embedding dimension $>2$ were previously known to exist only  in characteristic $2$ and in a special setting.  One of the main results of the present article is that when $A$ is an excellent local Noetherian domain of dimension $d>1$, then the dimension of the generic formal fiber of $A$ is $d-1$ if and only if $A$ is finitely dominated by an analytically ramified stable domain $R$ of embedding dimension $d$ (Theorem~\ref{generic}).  In Theorem~\ref{pre excellent} we show the ring $R$ arises via the construction of Heinzer, Rotthaus and Sally in a natural way from a prime ideal $P$ of $\widehat{A}$ such that $\widehat{A}/P$ has dimension $1$ and $P \cap A = 0$, and the embedding dimension of the ring $R$ is $1$ more than the embedding dimension of the ring $\widehat{A}_P$.  Thus when $A$ is excellent, or more generally, a $G$-ring, then this local ring $\widehat{A}_P$ is a regular local ring and hence the embedding dimension of $\widehat{A}_P$ is $d-1$. We obtain from this observation a bound on the embedding dimensions of the analytically ramified local Noetherian stable rings that finitely dominate $A$ (Corollary~\ref{excellent}).

The main results regarding the connection between stable rings and the generic formal fiber are in Sections 5 and 6, but since our main focus is on one-dimensional stable rings, and since stable rings are also of interest in non-Noetherian commutative ring theory (for some recent examples, see \cite{ElG, GP, KM, Mim, Sega, Zan2, Zan}),
 we include in Sections 3 and 4 characterizations of these rings in terms of their normalization and completion.

{\it Notation and terminology}.  All rings are commutative with identity. We use the following standard notions throughout the article.   
For a ring $R$, we denote by $\overline{R}$ the integral closure of $R$ in its total ring quotients.  Thus $\overline{R}$ is the {\it normalization} of $R$.  The ring $R$ has {\it finite normalization}  if $\overline{R}$ is a finite $R$-module.  When $R$ is quasilocal with maximal ideal $M$, we denote by $\widehat{R}$ the completion of $R$ in the $M$-adic topology.  The {\it embedding dimension} of the quasilocal ring $R$, denoted $\embdim R$, is the cardinality of a minimal generating set of $M$.

Let $A \subseteq S$ is an extension of rings, and let $L$ be an $S$-module.  An {\it $A$-linear  derivation} $D:S \rightarrow L$ is an $A$-linear mapping with $D(st) = sD(t) + tD(s)$ for all $s,t \in S$.  There exists an $S$-module $\Omega_{S/A}$, as well an   $A$-linear derivation {$d_{S/A}:S \rightarrow L$}, 
 such that  for each $A$-linear derivation $D:S \rightarrow L$, there   is a unique $S$-module homomorphism $\alpha: \Omega_{S/A} \rightarrow L$ with $D = \alpha \circ d_{S/A}$ \cite[pp.~191-192]{Ma}.  The module $\Omega_{S/A}$ is the module of {\it K\"ahler differentials} of the extension $A \subseteq S$ and $d_{S/A}$ is the {\it exterior differential} of this extension.

\section{Preliminaries on stable rings}

\label{(S)}

The terminology of stable ideals  originates with Lipman \cite{Lipman} and Sally and Vasconcelos \cite{SV}.  We leave the motivation for the terminology unexplained here, and refer instead to \cite{OlSurvey} for background on this class of rings.  
 An ideal $I$ of a ring $R$ is {\it stable} if it is
projective over its ring of endomorphisms.  A domain $R$ is  {\it stable} provided every nonzero ideal  is
stable.
An ideal $I$ of a quasilocal domain $R$ is stable if and only if 
 $I^2 = iI$ for some $i \in I$; if and only if $I$ is a principal ideal of $\End(I):= (I:_FI),$ where $F$ is the quotient field of $R$ (cf.~\cite{Lipman} and
 \cite[Lemma 3.1]{OlStructure}). 
%
It follows from the  Principal Ideal Theorem that a Noetherian stable domain has Krull dimension at most $1$.  
If  $R$ is a $2$-generator domain (meaning every ideal can be generated by $2$ elements),
then $R$ is a stable domain \cite{Bass}.  Conversely, when $R$ is a Noetherian stable domain with finite normalization, then $R$ is a $2$-generator domain \cite{DK, SVBull}.  
%
%
%
%
%
%
%
In \cite[Example 5.4]{SV}, 
Sally and Vasconcelos  showed that there do exist local Noetherian stable domains without
the 2-generator property (and hence without finite normalization). Their example was constructed using derivations and a 
method of Ferrand and Raynaud, and 
relied on a specific field of characteristic $2$. In 
\cite[(3.12)]{HLS}, Heinzer, Lantz  and Shah also observed that such
examples (again, using this specific field of characteristic $2$)
could be found for any choice of multiplicity. Note that for a local Noetherian stable ring $R$ with maximal ideal $M$, since $M^2 = mM$ for some $m \in M$,  the multiplicity and embedding dimension of $R$ agree.  In this article we see many  
 more such examples of analytically ramified one-dimensional local stable domains, and in all possible characteristics.  Our focus in this article is on the Noetherian case, but we develop characterizations in Sections 3 and 4 for the general one-dimensional case as well, since it requires little extra effort.    
   In \cite{OlbAR}, examples of non-Noetherian one-dimensional stable domains are given.   
 
To simplify terminology, we say  that a domain $R$  is a {\it bad
stable domain} if it is a quasilocal stable domain with  Krull dimension $1$ that does not have  finite normalization.  
The domain $R$ 
 is a {\it bad
$2$-generator ring} if $R$ is a $2$-generator local  ring that does not have  finite normalization.  
These rings are ``bad'' in the sense of Nagata's appendix in \cite{Na}, ``Examples of bad Noetherian rings,''  because they do not have finite normalization.  Thus a
 bad Noetherian stable domain is a  (necessarily one-dimensional) analytically ramified local Noetherian stable domain.  A bad $2$-generator domain is an analytically ramified $2$-generator domain.

In later sections we often use the following characterization of bad stable domains, which is a consequence of \cite[Corollary 4.3]{OlStructure},
 \cite[Corollary 2.5]{OlClass} and
\cite[Lemma 3.7]{OlRend}. 
We say that an extension $R \subseteq S$ of rings is {\it quadratic} if every $R$-submodule of $S$ containing $R$ is a ring; equivalently, $st \in sR + tR + R$ for all $s,t \in S$.

\begin{prop} \label{ar stable char} A  quasilocal domain $R$, not a DVR, is 
 a bad stable domain if and only if $\overline{R}$ is a DVR such that $\overline{R}/R$ is a divisible $R$-module and  $R \subseteq \overline{R}$ is a quadratic extension.      
\end{prop}

It is also useful to note that when $R$ is a quasilocal domain and $\overline{R}$ is a DVR, then $\overline{R}/R$ is a divisible $R$-module if and only if $\overline{R}$ tightly dominates $R$.  
Thus a quasilocal domain $R$, not a DVR, is 
 a bad stable domain if and only if $R \subseteq \overline{R}$ is a quadratic extension and $\overline{R}$ is a DVR that tightly dominates $R$.

\section{The completion of a bad stable ring}

Let $R$ be a ring, and let $C$ be a multiplicatively closed set of nonzerodivisors of $R$.  An $R$-module $L$ is {\it $C$-divisible} if for each $c \in C$ and $\ell \in L$, there exists $\ell' \in L$ such that $\ell = c\ell'$. The module $L$ is {\it $C$-torsion} provided that for all $\ell \in L$, there exists $c \in C$ such that $c\ell =0$.  When $R \subseteq S$ is an extension of rings and $L$ is an $R$-module, we say that $L$ {\it admits an $S$-module structure} if there exists an $S$-module structure on $L$ extending its $R$-module structure.

\begin{lem} \label{admits} \label{torsion admits} Let $R \subseteq S$ be an extension of rings, and let $C$ be a multplicatively closed subset of $R$ consisting of nonzerodivisors of $S$.    Suppose that $S/R$
is {{\it C}}-divisible, and let $T$ be a ${{{{\it C}}}}$-torsion
$R$-module.

\begin{itemize}

\item[(1)]  The $R$-module $T$ admits an  $S$-module
structure if and only if for every $t \in T$, it is the case that
$(0:_{R} t)S \cap R = (0:_{R} t)$.

\item[(2)]  If $T$ admits an $S$-module structure, then
this   structure is unique and is given for each $s \in S$ and $t
\in T$  by $s \cdot t = rt$, where  $r$ is any member of $R$ such
that $s-r \in (0:_R t)S$ (and such a member $r$ of $R$ must exist).


\end{itemize}

\end{lem}

\begin{proof}
(1)  If $T$ admits an $S$-module structure $\ast$, then, using the fact that $\ast$ extends the $R$-module
structure on $T$, it is easy to check that for all $t \in T$,  it is the case that $(0:_{R} t)S \cap R
= (0:_{R} t)$.
%
Conversely, suppose  that for all $t \in T$,  it is the case that $(0:_{R} t)S \cap R
= (0:_{R} t)$. Let $s \in S$ and $t \in T$.  Then since $T$
is ${{{{\it C}}}}$-torsion, there exists $c \in {{{{\it C}}}}$ such
that $ct = 0$. Thus since $S = R + cS$, we have  that $S = R +
(0:_{R} t)S$, and there exists $r \in R$ such that $s -r \in (0:_{R}
t)S$. We define $s \cdot t = rt$.  We omit the calculations, but it is straightforward to check that this defines an $S$-module structure on $T$.  

(2)
Suppose that $\ast$ denotes an $S$-module structure on $T$ that
extends the $R$-module structure on $T$.  We show that the
operations $\ast$ and $\cdot$, where $\cdot$ is defined as in (2),
induce the same $S$-module structure on $T$.  Let $s \in S$ and $t
\in T$. Since $T$ is ${{{{\it C}}}}$-torsion, there exists $c \in
{{{{\it C}}}}$ such that $ct = 0$.  So since by assumption, $S = R +
cS$, we may choose $r \in R$,
 and $\sigma \in S$ such that
$s - r = c \sigma$. Then, using the fact that $\ast$
extends the $R$-module structure on $T$, we have
$s \ast t = (r +c \sigma) \ast t = 1 \ast rt + \sigma
\ast (ct)  = rt,$ which proves that $s
\ast t = s \cdot t$.
\end{proof}

Recall from Section 2 that a ring extension $R \subseteq S$ is quadratic  if every $R$-submodule of $S$ containing $R$ is a ring.   

\begin{lem} \label{start} Let $R \subseteq S$ be an extension of rings, and suppose that there exists a mulitplicatively closed subset $C$  of $R$ consisting of  nonzerodivisors in $S$ such that $S/R$ is $C$-torsion and $C$-divisible.
 Then  the following statements are equivalent.

\begin{itemize}

\item[(1)] $R \subseteq S$ is a quadratic extension of rings. \index{quadratic extension}

\item[(2)]  For all $s \in S$, $(R:_R s) = (R:_R s)S \cap R$.

\item[(3)]  $S/R$ admits an $S$-module structure. \index{admits an $S$-module structure}


\item[(4)]  There exists an
$S$-module $T$ and a  derivation $D:S \rightarrow T$ with  $R =
\Ker D$.

\item[(5)] The mapping $S/R \rightarrow \Omega_{S/R}$ induced by $d_{S/R}$ is an   isomorphism of $R$-modules.

\item[(6)]  For all $c \in {{{{\it C}}}}$, $(R \cap cS)^2 \subseteq cR$.



\end{itemize}

\end{lem}

\begin{proof}
(1) $\Rightarrow$ (2) Let $s \in S$, and suppose that $x \in (R:_R
s)S \cap R$. We claim that $x \in (R:_R s)$. Since $S/R$ is ${{{{\it
C}}}}$-torsion, there exists $c \in {{{{\it C}}}}$ such that $cs \in
R$. Also, since $S = R + cS$, we have $(R:_Rs)S = (R:_Rs) + c(R:_Rs)S \subseteq (R:_Rs) + cS$, so that 
 there exist $a \in (R:_R s)$ and
$\sigma \in S$ such that $x = a +c\sigma$.  By (1), $s \sigma \in sR
+ \sigma R + R$, so $cs \sigma   \in   csR + c\sigma   R + cR$.
Thus, since $cs \in R$, $c\sigma = x - a \in R$ and $c \in R$, it
follows that $c\sigma \in (R:_R s)$.  By assumption, $a \in (R:_R
s)$, so we conclude that $x = a + c\sigma \in (R:_R s)$, which
proves (2).

(2) $\Rightarrow$ (3) Since  $S/R$ is ${{{{\it C}}}}$-divisible and
$C$-torsion, and by (2), $(R:_R s)S \cap R = (R:_R s)$ for all $s
\in S$, we have by Lemma~\ref{admits}(1) that $S/R$ admits an
$S$-module structure.

(3) $\Rightarrow$ (4) Define a mapping $D:S \rightarrow S/R$ by
$D(s) = s+R$ for all $s \in S$.  By (3), $S/R$ has an $S$-module
structure, which by Lemma~\ref{admits}(2) must be  given by the
operation $\cdot$ defined in the lemma. We claim that with this $S$-module
structure on $S/R$, $D$ is a derivation, for once this is proved,
(4) follows at once.
Let $s_1,s_2 \in S$.  Then since $S/R$ is ${{{{\it C}}}}$-divisible
and ${{{{\it C}}}}$-torsion, there exist $r_1,r_2 \in R$,
$\sigma_1,\sigma_2 \in S$, $c_1 \in (R:_R s_2)$ and $c_2 \in (R:_R
s_1)$ such that $s_1= r_1 + c_1\sigma_1$ and $s_2 =
r_2+c_2\sigma_2$. Then applying the definition of $\cdot$ we have:
$$s_1 \cdot D(s_2) + s_2 \cdot D(s_1) =
s_1 \cdot (s_2+R) + s_2 \cdot (s_1+R) = r_1s_2 + r_2s_1 +R.$$ Hence
to prove that $D(s_1s_2) = s_1 \cdot D(s_2) + s_2 \cdot D(s_1)$, it
suffices to show that $D(s_1s_2) = r_1s_2 + r_2s_1 + R$. Now
\begin{eqnarray*}D(s_1s_2)\: \:  = \:\: s_1s_2 +R
& = & (r_1+c_1\sigma_1)(r_2 + c_2\sigma_2) + R \\
\: & = & r_1c_2\sigma_2 + r_2c_1\sigma_1 + c_1c_2\sigma_1\sigma_2+R
\\
\: & = & r_1(s_2 - r_2) + r_2(s_1-r_1) +  c_1c_2\sigma_1\sigma_2 + R \\
\: & = & r_1s_2 + r_2s_1 + c_1c_2\sigma_1\sigma_2 + R.
\end{eqnarray*}
Therefore,   we need only verify that $c_1c_2\sigma_1\sigma_2 \in
R$. To this end, observe that since $c_2 \in (R:_R s_1)$,  then $0 +R=
c_2s_1 +R = c_2(r_1+c_1\sigma_1) + R = c_2c_1\sigma_1 + R,$ so
$c_1c_2\sigma_1 \in R$.
 Hence, using the fact that $\cdot$ extends the $R$-module
 structure  on $S/R$, as well as the fact that
 $c_1$, $c_2$, $c_1s_2$ and $c_1c_2\sigma_1$ are all members of $R$, we have:
\begin{eqnarray*} 0+R  &= & \sigma_1 \cdot (c_1s_2 + R) \:\: =
\:\: c_1\sigma_1 \cdot (r_2 +
c_2\sigma_2 + R) \\
\: &= &  c_1\sigma_1 \cdot (c_2\sigma_2 +R) \:\: = \:\:
c_1c_2\sigma_1 \cdot (\sigma_2 + R) \:\: = \:\: c_1c_2\sigma_1\sigma_2 + R
\end{eqnarray*}  Therefore, $c_1c_2\sigma_1\sigma_2 \in R$, which
proves that $D(s_1s_2) = s_1 \cdot D(s_2) + s_2 \cdot D(s_1)$.
Clearly, $D(s_1+s_2) = D(s_1) + D(s_2)$, so $D$ is a derivation.

(4) $\Rightarrow$ (1) Let $x,y \in S$.  We show that $xy \in xR + yR
+ R$. Since $S/R$ is ${{{{\it C}}}}$-torsion, there exists  $c \in
{{{{\it C}}}}$ such that $cx,cy \in R$.  Thus, since $c \in R = \Ker
D$, we have $cD(x) = D(cx) = 0$ and $cD(y) = D(cy) = 0$.  Also,
since $S/R$ is ${{{{\it C}}}}$-divisible, there exist $a,b \in R$
such that $x+a,y+b \in cS$. Hence  $(x+a)  D(y) =0$ and $(y+b)
 D(x) =0$. Therefore, since also $D(a) = D(b) = 0$, we have:
\begin{eqnarray*}
D(xy+ya+xb) & = & D(xy) + D(ya) + D(xb) \\
\: & = & x  D(y)+ y  D(x) + a  D(y)  + b  D(x)   \\
 \: & = & (x+a)  D(y) + (y+b) D(y) \: \: = \: \:  0.
 \end{eqnarray*} Thus $xy + ya + xb \in \Ker D = R$, whence $xy \in xR + yR + R$.

(4) $\Rightarrow$ (5)  Assuming (4), there exists an $S$-module
homomorphism $\alpha:\Omega_{S/R} \rightarrow T$ such that $\alpha
\circ d_{S/R} = D$ (see Section 1).  Thus $\Ker
d_{S/R} \subseteq \Ker D = R$. But $d_{S/R}$ is an $R$-linear
derivation, so it must be that $\Ker d_{S/R}  = R$. Thus to verify (5), it suffices to show that 
 $d_{S/R}$ is an onto mapping.  
To this end, note that since $d_{S/R}(S)$ generates $\Omega_{S/R}$ as an $S$-module, there
 exist $s_1,\ldots,s_n,x_1,\ldots,x_n \in S$ such that $y
= \sum_{i=1}^n s_i d(x_i)$.  Choose $c \in {{{{\it C}}}}$ such
that $cx_1 ,\ldots,cx_n \in R$.  Then since $S = \Ker d_{S/R}  + cS$, we may for
each $i$ write $s_i = a_i + c\sigma_i$, where $a_i \in \Ker d_{S/R}$ and
$\sigma_i \in S$.  Thus, since $a_1,\ldots,a_n \in \Ker d_{S/R}$,  we have: $$y = \sum_{i}d_{S/R}(a_ix_i) +
\sum_{i}\sigma_i cd_{S/R}(x_i)  =
d_{S/R}(\sum_i a_ix_i).$$ Therefore, $d_{S/R}$ maps  onto $\Omega_{S/R}$, and  $\Omega_{S/R}$ and $S/R$ are isomorphic as
$R$-modules.

(5) $\Rightarrow$ (3)  Since $\Omega_{S/R}$ is an $S$-module, this
is clear.

(1) $\Rightarrow$ (6) Let $c \in C$, and let $s_1,s_2 \in S$ such that $cs_1,cs_2 \in R$.  It suffices to show that $c^2s_1s_2 \in cR$.  By (1), $s_1s_2 \in s_1R + s_2R + R$, so that $c^2s_1s_2 \in c(cs_1)R + c(cs_2)R + c^2R \subseteq cR$, as claimed.


(6) $\Rightarrow$ (2)  Let $s \in S$, and let $x \in (R:_R s)S \cap
R$.  Since $S/R$ is $C$-torsion, there exists $c \in C \cap (R:_Rs)$.  Also,  since $S/R$ is  $C$-divisible, we have $S = R + cS$, and hence $(R:_R s)S = (R:_Rs) + cS$.  
Thus there exist 
 $a \in (R:_R s)$ and
$\sigma \in S$ such that $x = a+c\sigma$, and  in order to show that
$x \in (R:_R s)$, it suffices to prove that $c\sigma s \in R$. Since
$c \sigma = x -a \in R$ and $cs \in R$, we have  $(c\sigma)(cs) \in
(cS \cap R)^2$, so that by (6), $c^2\sigma s \in cR$.  Thus since
$c$ is a nonzerodivisor in $S$,  $c \sigma s \in R$, as claimed.
\end{proof}

Theorem~\ref{complete char} shows that the completion of a bad stable domain has a prime ideal whose square is $0$ and whose residue ring is a DVR.  Ideals in such rings have principal reductions of reduction number at most $1$:

\begin{lem} \label{stable null}  Let  $N$ be an ideal of the ring $R$
such that $N^2 = 0$.  If $I$ is an ideal of $R$ whose image in $R/N$
is a principal ideal generated by $x+N$ for some $x \in I$, then
$I^2 = xI$. 
\end{lem}

\begin{proof}
 To prove that $I^2 = xI$, it suffices to show that
for all $y,z \in I$, $yz \in x(x,y,z)R = (x^2,xy,xz)R$.  Let $y,z
\in I$. Write $y = xr_1 + n_1$ and $z = xr_2 + n_2$ for some
$r_1,r_2 \in R$ and $n_1,n_2 \in N$.  Observe that since $N^2 = 0$,
we have:
\begin{center}
 $yz = x^2r_1r_2 + xr_1n_2 + xr_2n_1$ \ \ \ \ \
 $xy = x^2r_1 + xn_1$ \ \ \ \ \
 $xz = x^2r_2 + xn_2$.
\end{center}
{\noindent}The following calculation now shows that $yz \in
(x^2,xy,xz)R$:
\begin{eqnarray*} -r_1r_2(x^2) + r_2(xy) + r_1(xz) &
= & -r_1r_2x^2 + r_2(x^2r_1+xn_1) + r_1(x^2r_2 + xn_2) \\
 \: & = & -r_1r_2x^2 + r_1r_2x^2 + r_2xn_1 + r_1r_2x^2 + r_1xn_2 \\
\: & = & r_1r_2x^2 + r_2xn_1 + r_1xn_2 \:\: = \:\:  yz.
\end{eqnarray*}
Thus $yz \in x(x,y,z)R \subseteq xI$, which proves that $I^2 = xI$.
\end{proof}



We give in the next theorem the characterization of the completion of a bad stable ring, but there is a small subtlety in how it is phrased.  The theorem is proved for not-necessarily-Noetherian rings, and because of this the domain $R$ need not be separated in the ${\ff m}$-adic topology.  (A bad stable domain is easily seen to be separated, but in the converse, in showing that the given domain is a bad stable domain, separation is needed to guarantee that $R \rightarrow \widehat{R}$ is a flat embedding.)  Thus we state the theorem for the {\it completion of $R$ in the  ideal topology (or $R$-topology)}: $\wt{R}:=\lim_{\leftarrow}R/rR$, where $r$ ranges over all nonzero elements  of $R$.  For properties and applications of this completion, see \cite{FS,M1,M2}.   When $R$ is a one-dimensional Noetherian ring, which is the main case of  interest in this article, then $\wt{R}$ is isomorphic as an $R$-algebra to $\widehat{R}$, and we deduce in Corollary~\ref{Noetherian completion} an ${\ff m}$-adic version of the theorem in the Noetherian case.

\begin{thm} \label{complete char} A one-dimensional quasilocal domain is a bad stable ring  if and only if there is a nonzero prime ideal $P$ of $\wt{R}$ such that $P^2 = 0$ and $\wt{R}/P$ is a DVR.  
\end{thm}

\begin{proof}  
Let  $R$ be a bad stable ring, and let $r$ be a nonzero nonunit in $R$.      Then since $R$ has dimension $1$, $Q:=R[1/r]$ is the quotient field of  $R$, and hence $Q$ is  a countably generated $R$-module.  By a lemma of Auslander, this in turn implies that  $Q$ has projective dimension $1$ as an $R$-module \cite[Lemma VI.2.6, p.~203]{FS}.    Since we are working with the completion of $R$ in the ideal topology, and $Q$ has projective dimension $1$, there is a surjective ring homomorphism   $\psi: 
 \wt{R} \rightarrow
{(\overline{R})}^\sim $, where 
 $(\overline{R})^\sim$ is the completion of $\overline{R}$ in the ideal topology of $\overline{R}$ \cite[Theorem 2.9, p.~21]{M1}.
 Since $\overline{R}$ is a DVR that is integral over $R$, it follows that the $R$-topology and $\overline{R}$-topology on $\overline{R}$ agree, so with $S:=\lim_{\leftarrow}\overline{R}/r\overline{R}$, where $r$ ranges over the nonzero elements of $R$,   
 the canonical mapping $\phi:\wt{R} \rightarrow S$  is a surjection.  
  Now $S$, as completion of a DVR, is a DVR, and since $R$ is not a DVR, $P:=\Ker \phi \ne 0$, so we need only show that $P^2 = 0$.  To this end, let $x,y \in P$.  Write $x = \l x_r + rR \r$ and $y = \l y_r + rR\r$.  (We are viewing $\wt{R}$ as a subring of $\prod_{r}R/rR$ and $S$ as a subring of $\prod_R \overline{R}/r\overline{R}$.)  Then $0 = \phi(x) = \l x_r + r\overline{R} \r$, so that for each $r$, $x_r \in r\overline{R} \cap R$.  Similarly, $y_r \in r\overline{R} \cap R$, and hence by Lemma~\ref{start}(6), $x_ry_r \in (r\overline{R} \cap R)^2 \subseteq rR$, proving that $xy = 0$, and hence that $P^2 =0$. 

Conversely, suppose that there is a nonzero prime ideal $P$ of $\wt{R}$ such that $P^2 =0$ and $\wt{R}/P$ is a DVR.  If $I$ is a nonzero proper ideal of $R$, then since the image of $I\wt{R}$ in $\wt{R}/P$ is a principal ideal, we have by Lemma~\ref{stable null} that there exists $x \in 
I\wt{R}$ such that $I^2\wt{R} =xI\wt{R}$.  We show that this implies 
that there exists $i \in I$ such that $I^2 = iI$.  
With $\lambda:R \rightarrow \wt{R}$  the canonical map, the $R$-module
$\wt{R}/\lambda(R)$ is  divisible \cite[Theorem 2.1, p.~11]{M1}.  Thus   $x \in  \lambda(R)+x^2\wt{R}$, and  there exist $y \in \wt{R}$ and $i \in R$ with  $x(1-xy) =\lambda(i)$.  Since $x$ is not a unit in the quasilocal ring $\wt{R}$, it follows that $x\wt{R} = i\wt{R}$.  Moreover, since $\wt{R}$ is a flat $R$-module \cite[Corollary 2.6]{M1}, $i \in \lambda^{-1}(I\wt{R}) = I$. This shows that there exists $i \in I$ such that $I^2\wt{R}= iI\wt{R}$.  Again by flatness, $I^2 = iI$.   Thus   
 $I$ is  a stable ideal of $R$, and $R$ is a stable domain.
 Moreover, $R$ is a bad stable domain.  For suppose $R$ has finite normalization.  Then $R$ is a one-dimensional local Noetherian domain, so that  $\widetilde{R} = \widehat{R}$.  Thus since $R$ has  finite normalization, then $\widetilde{R}$ is reduced, contrary to the fact that $P$ is a nonzero ideal of $\widetilde{R}$ with $P^2 = 0$.      
\end{proof}

Restricting to Noetherian rings, the ${\ff m}$-adic and ideal completions agree, so we obtain

\begin{cor} \label{Noetherian completion} A local Noetherian domain is a bad stable domain if and only if there is a nonzero prime ideal $P$ of $\widehat{R}$ such that $P^2=0$ and $\widehat{R}/P$ is a DVR. \qed
\end{cor}

\section{Characterizations of bad stable rings}

\label{stable chars}

\label{2: application: bad stable rings}

An $R$-module $K$ is {\it uniserial} if the set of all
$R$-submodules of $K$ is linearly ordered with respect to inclusion.

\index{Artinian uniserial module} \begin{lem} \label{uniserial quadratic} Let $R \subseteq S$ be an integral  extension of rings, with $S$ a quasilocal ring.
  If $S/R$ is an Artinian uniserial $R$-module,
then $R \subseteq S$ is a quadratic extension.
\end{lem}

\begin{proof}
 Let $\cal{S}$  be the set of
all $R$-submodules of $S$ containing $R$, and let $\cal{R}$ be the
set of all $B \in {\cal S}$ such that $B$ is a ring.  We claim that $ {\cal{S}}= {\cal{R}}$, and hence that $R \subseteq S$ is a quadratic extension.
 Since $S/R$ is Artinian and uniserial, the set ${\cal S}$ is
well-ordered with respect to inclusion, so we may use transfinite
induction to show that ${\cal S} = {\cal R}$. In particular, to
prove that ${\cal S} = {\cal R}$, it suffices to show that for each
$F \in {\cal S}$,
$$\{E \in {\cal S}: E \subsetneq F\} \subseteq {\cal R} \: \: \:
\Rightarrow \: \: \: F \in {\cal R}.$$  Let $F \in {\cal S}$, and
suppose that $\{E \in {\cal S}: E \subsetneq F\} \subseteq {\cal
R}$. Let $E'$ be the union of all the $E \in {\cal S}$ such that $E
\subsetneq F$.  Then since each such $E$ is a ring and ${\cal S}$ is a chain, $E'$ is also a
ring.  Thus if $E' = F$, then the claim that $F \in {\cal R}$ is
proved.  Otherwise, suppose that $E' \subsetneq F$. Since $E'$ is a
ring and $S$ is quasilocal and integral over $R$, it follows that $E'$
has a unique maximal ideal $M$.  Now $F/E'$ is necessarily a simple
$E'$-module, so $MF \subseteq E'$.  In fact, $MF \subseteq M$, since
otherwise the maximality of $M$ in $E'$ implies $1 \in MF \subseteq MS$, a contradiction to the fact that
$S$ is integral over $E'$.   Therefore, $F \subseteq
 (M:_S
M)$.
 Yet $F/E' \subseteq (M:_S M)/E'$ is a containment of vector spaces
  over $R/(M \cap R)$, so the fact that
$S/R$, and hence $S/E'$, is a uniserial $R$-module implies that $(M:_S M)/E'$ is a simple
$R$-module, and this forces $F = (M:_SM)$.  Therefore, $F$  is a
ring, and $F \in {\cal R}$, which proves that ${\cal R} = {\cal S}$.
This shows that  $R \subseteq S$ is a quadratic extension.
\end{proof}


\begin{thm} \label{ar stable thm}  The
following statements are equivalent for a quasilocal domain $R$ with  quotient field $F$. 

\begin{itemize}

\index{stable domain!bad}
\item[(1)]  $R$ is bad  stable domain.


\index{DVR}
\item[(2)]  ${\overline{R}}$ is a DVR and ${\overline{R}}/R \cong \bigoplus_{i \in I}F/{\overline{R}}$ as $R$-modules
for some index set $I$.

\index{Artinian uniserial module}
\item[(3)]  ${\overline{R}}$ is a DVR and ${\overline{R}}/R$ is a direct sum of divisible
Artinian uniserial $R$-modules.

\end{itemize}
\end{thm}

\begin{proof}
(1) $\Rightarrow$ (2)
 Since by Proposition~\ref{ar stable char},  ${\overline{R}}/R$ is a torsion divisible $R$-module and  $R
\subseteq {\overline{R}}$ is a quadratic extension, we may apply
Lemma~\ref{start} to obtain that ${\overline{R}}/R$ is an ${\overline{R}}$-module.
Now since $\overline{R}$ and $R$ share the same quotient field, as a divisible $R$-module,  ${\overline{R}}/R$
is also a divisible ${\overline{R}}$-module.  Also, by 
Proposition~\ref{ar stable char}, ${\overline{R}}$ is a DVR, and hence every divisible ${\overline{R}}$-module is an
injective ${\overline{R}}$-module.  In particular, ${\overline{R}}/R$ is an injective
${\overline{R}}$-module.  Let $N$ denote the maximal ideal of ${\overline{R}}$. Since ${\overline{R}}$ is a
Noetherian domain, every injective ${\overline{R}}$-module is a direct sum of
indecomposable ${\overline{R}}$-modules \cite[Theorem 2.5]{MatlisNoetherian}, and since also ${\overline{R}}$ is a DVR,
 a torsion ${\overline{R}}$-module is injective and indecomposable if and
only if it is isomorphic to $F/{\overline{R}}$ (for example, combine \cite[Theorem 4]{M3} and
\cite[Theorem 4.5]{M1}). Therefore, since
${\overline{R}}/R$ is a torsion injective ${\overline{R}}$-module, we have that ${\overline{R}}/R$ is isomorphic as an ${\overline{R}}$-module, and hence an $R$-module, to a direct sum of copies of $F/{\overline{R}}$.

(2) $\Rightarrow$ (3)  By (2), ${\overline{R}}/R$ is a direct sum of $R$-modules, say ${\overline{R}}/R = \bigoplus_{i \in I}B_i/R$, where each $B_i/R \cong F/{\overline{R}}$ as $R$-modules.  Since ${\overline{R}}$ is a DVR, $F/{\overline{R}}$ is a divisible Artinian uniserial ${\overline{R}}$-module.  Thus to show that for each $i$, $B_i/R$ is a divisible Artinian uniserial $R$-module, it suffices to show that every $R$-submodule of $F/{\overline{R}}$ is also an ${\overline{R}}$-submodule. Let $A/\overline{R}$ be an $R$-submodule of $F/\overline{R}$.  We  show that $A\overline{R} \subseteq A$.    Let $a \in A$ and $x \in \overline{R}$.  Since $F/\overline{R}$ is a torsion $R$-module, there exists $0 \ne r \in R$ such that $ra \in \overline{R}$.  Since $ \overline{R}/R$ is divisible, there exist $s \in R$ and $y \in \overline{R}$  with $x = s + ry$.  Therefore, since $\overline{R} \subseteq A$, then  $xa = sa + (ra)y \in A$, which proves that every $R$-submodule of $F/\overline{R}$ is an $R$-module.

(3) $\Rightarrow$ (1)  Write ${\overline{R}}/R  = \bigoplus_{i \in I}B_i/R$, where
for each $i \in I$, $B_i/R$ is a divisible Artinian uniserial
$R$-module.
Fix $i
\in I$, and let $B = \sum_{j \ne i}B_j$.    We claim first that $B_i/R$ is
an ${\overline{R}}$-module. Indeed, since ${\overline{R}}= B_i + B$ and $R = B_i \cap B$, we
 have $B_i/R \cong (B_i+B)/B = {\overline{R}}/B$ as $R$-modules.  Now $B_i/R$ is by
 assumption a divisible Artinian uniserial $R$-module, so ${\overline{R}}/B$ is also a divisible
 Artinian uniserial $R$-module, and hence also a divisible Artinian uniserial
 $B$-module (where divisibility as a $B$-module follows from the fact that $B/R$ is a torsion $R$-module).  Observe that $B$ is a ring.  For let $a,b \in B$.  Then there exists $0 \ne r \in R$ such that $rb \in R$.  Since $B/R$ is a divisible $R$-module, there exist $c \in B$ and $s \in R$ with $a = rc + s$.  Hence $ab = (rc+s)b = (rb)c + sb \in B$, proving that $B$ is a ring.  
  Therefore, by Lemma~\ref{uniserial quadratic}, $B \subseteq {\overline{R}}$ is a quadratic extension, and hence by Lemma~\ref{start},  ${\overline{R}}/B$ is an ${\overline{R}}$-module.   But then, since $B_i/R \cong {\overline{R}}/B$, we have that $B_i/R$
  admits an ${\overline{R}}$-module structure.  Therefore, ${\overline{R}}/R$, as a direct sum of
  the $B_i/R$, also admits an ${\overline{R}}$-module structure, and by
  Lemma~\ref{start}, $R \subseteq {\overline{R}}$ is a quadratic extension.
   Thus since ${\overline{R}}$ is a DVR, we have by Proposition~\ref{ar stable char}
   that $R$ is a bad stable domain.
\end{proof}


\begin{cor} \label{stable Noetherian cor} Let $R$ be a quasilocal
domain  with quotient field $F$, and let $n>1$.  Then the following
statements are equivalent.

 \begin{itemize}

\index{stable domain!bad Noetherian}
\index{stable domain!embedding dimension of}
 \item[(1)] $R$ is a bad Noetherian stable
 domain of  embedding dimension $n$.

\index{DVR} \index{Artinian uniserial module}
      \item[(2)]  ${\overline{R}}$ is a DVR and ${\overline{R}}/R$ is a direct sum of $n-1$ divisible Artinian uniserial $R$-modules.

          \item[(3)]  ${\overline{R}}$ is a DVR and ${\overline{R}}/R \cong \bigoplus_{i=1}^{n-1} F/{\overline{R}}$ as $R$-modules.

\end{itemize}
\end{cor}

\begin{proof}
(1) $\Rightarrow$ (2)  Assume (1), and let $M$ denote the maximal ideal of $R$.  By Theorem~\ref{ar stable thm}, ${\overline{R}}$ is  a DVR and
${\overline{R}}/R$ is a direct sum of divisible Artinian uniserial $R$-modules,
say ${\overline{R}}/R = \bigoplus_{i \in I}B_i/R$, where each $B_i/R$ is a
divisible Artinian uniserial $R$-module.  For each $i$, since $B_i/R$ is Artinian,
 there exists
an $R$-submodule $D_i$ of $B_i$ such that $R \subseteq D_i \subseteq
B_i$ and $D_i/R$ is a simple $R$-module.  Hence for each $i$,  $MD_i
\subseteq R$, and we have $D_i \subseteq M^{-1}:=(R:_F M)$.  We
claim that $M^{-1}/R = \bigoplus_{i \in I}D_i/R$.  For each $i \in I$,
let $\pi_i$ be the projection of ${\overline{R}}/R$ onto $B_i/R$.  Then for each
$i$, $D_i/R \subseteq \pi_i(M^{-1}/R) \subseteq B_i/R$.  Now
$\pi_i(M^{-1}/R)$ is an $R/M$-vector space, so since $B_i/R$ is
uniserial, it must be that $\pi_i(M^{-1}/R)$ has dimension $1$ as an $R/M$-vector space, and
hence is equal to $D_i/R$.  Therefore,  $\bigoplus_{i \in I}D_i/R
\subseteq M^{-1}/R \subseteq \bigoplus_{i \in I}D_i/R$, and since each
each $D_i/R$ is a simple $R$-module, we conclude that the vector space dimension
of $M^{-1}/R$ and the cardinality $|I|$ of $I$ are the same.
    Now since $R$ is not a DVR, we have $MM^{-1} = M$, then  $M^{-1} = (M:_F M)$.  The fact that $M$ is stable then implies that
     $M = mM^{-1}$ for some $m \in M$, and hence
     $M/M^2 = mM^{-1}/mM \cong M^{-1}/M$.  The embedding dimension $n$ of $R$ is thus the same as the dimension of the $R/M$-vector space $M^{-1}/M$.
      Therefore,
       the dimension of the $R/M$-vector space $M^{-1}/R$ is $n-1$, which forces $|I| = n-1$.

 (2) $\Rightarrow$ (1)  By Theorem~\ref{ar stable thm}, $R$ is bad stable domain.  Thus, as we saw above in the conclusion of the proof of (1) $\Rightarrow$ (2), the number of generators needed for $M$ is one more than  the dimension of $M^{-1}/R$.  Moreover, as in the proof of (1) $\Rightarrow$ (2),  this dimension is the same as the number of divisible uniserial Artinian $R$-modules in the decomposition of ${\overline{R}}/R$, which by assumption is $n-1$.

 (2) $\Leftrightarrow$ (3) This is clear in view of the proof of Theorem~\ref{ar stable thm}.
 \end{proof}

The last theorem of this section shows that wherever a bad stable domain $R$  can be found, it is possible to find bad Noetherian stable overrings of $R$ with prescribed embedding dimension, as long as that embedding dimension does not surpass the embedding dimension of $R$.

\begin{thm} \label{convex} Let $R$ be a bad stable domain of embedding dimension  $n$  (where possibly $n$ is infinite).  Then:

\begin{itemize}

\item[(1)]  If $n$ is finite (equivalently, $R$ is Noetherian), then for every  ring $S$  with $R \subseteq S \subsetneq \overline{R}$,
the ring $S$ is a bad Noetherian stable domain with embedding dimension at most $n$.

\item[(2)]  For  every integer $k$ with $2 \leq k \leq n$, there is a bad Noetherian stable domain $S$ of embedding dimension $k$ with $R \subseteq S \subseteq \overline{R}$ and such that $S/R$ is isomorphic to a direct sum of copies of $F/\overline{R}$.

\end{itemize}
\end{thm}

 \begin{proof} (1)  
  Let $S$ be a ring with $R \subseteq S \subsetneq \overline{R}$.   Since $S \subseteq \overline{R}$ is a quadratic extension, $\overline{R}$ is a DVR and $\overline{R}/S$ is a divisible $S$-module, Proposition~\ref{ar stable char} implies that $S$ is a bad stable domain.
 Since $R$ is a one-dimensional local Noetherian domain, the multiplicity of every local overring of $R$  is bounded above by the multiplicity of $R$ \cite[Theorem 2.1]{Naga}.  Since for a stable ring, embedding dimension and multiplicity agree, statement (1) now follows.

 (2)
  Assume that $2 \leq k < n$.    Since $R \subseteq \overline{R}$ is a quadratic extension, Theorem~\ref{ar stable thm} implies there exist $R$-subalgebras $S_1,\ldots,S_{k-1},S$ of $\overline{R}$ such that $$\overline{R}/R =S_1/R \oplus \cdots \oplus S_{k-1}/R \oplus S/R,$$ where $S/R$ is isomorphic to a direct sum of copies of $F/\overline{R}$, and
 for each $i$, $S_i/R \cong F/\overline{R}$.
 We claim that $S$ is a bad Noetherian stable domain with embedding dimension $k$.
  Now  $$\overline{R}/S \cong (\overline{R}/R)/(S/R) \cong S_1/R \oplus \cdots \oplus S_{k-1}/R \cong \bigoplus_{i=1}^{k-1}F/\overline{R},$$ where these  isomorphisms are as $R$-modules.
 Thus there is an isomorphism of $R$-modules: $$\alpha:\overline{R}/S \rightarrow \bigoplus_{i=1}^{k-1}F/\overline{R}.$$  To prove that $S$ is a bad Noetherian stable domain with embedding dimension $k$, it suffices by Corollary~\ref{stable Noetherian cor} to show that $\alpha$ is an $S$-module homomorphism.
 Now $\alpha$ induces an $\overline{R}$-module structure on $\overline{R}/S$, which by Lemma~\ref{admits} is the unique $\overline{R}$-module structure extending the $R$-module structure on $\overline{R}/S$ and must be given as in the lemma.   Thus we need only verify that this $\overline{R}$-module structure extends the $S$-module structure on $\overline{R}/S$.
  Let $s \in S$ and $v \in \overline{R}$.   Then as in Lemma~\ref{admits}, $s \cdot (v + S) = rv + S$, where $s-r \in (S:_R v)\overline{R}$.  But $S \subseteq \overline{R}$ is a quadratic extension, so by Lemma~\ref{start}, $(S:_R v)\overline{R} \cap S \subseteq (S:_S v)\overline{R} \cap S = (S:_S v)$, and hence $s-r \in (S:_S v)$.  This then implies that  $sv + S = rv +S$, which proves that the $\overline{R}$-module structure on $\overline{R}/S$ extends the $S$-module structure.  Hence $S$ is a bad Noetherian stable domain of embedding dimension $k$.
  \end{proof}

\section{Representation and embedding dimension of   stable rings}





Let $A$ be a local Noetherian domain with maximal ideal ${\ff m}$, and suppose that $I$ is an ideal of $\widehat{A}$ such that every associated prime $P$ of $I$  satisfies
$A \cap P = 0$, where $A$ is identified with its image in $\widehat{A}$.  Then since $I \cap A = 0$,
 the canonical mapping $A \rightarrow \widehat{A}/I$ is an embedding, and we can identify $A$ with its image in $\widehat{A}/I$.  Under this identification, since the associated primes of $I$ contract to $0$ in $A$, it follows that the nonzero elements of $A$ are nonzerodivisors in $\widehat{A}/I$.  Therefore, the quotient field $F$ of $A$ embeds into the total quotient ring of $\widehat{A}/I$, and hence we may consider the ring $F \cap (\widehat{A}/I)$.    More precisely,
 $$F \cap (\widehat{A}/I) = \left\{ \frac{a}{b} \in F: a,b \in A, b \ne 0 {\mbox{ and }} a \in b\widehat{A} + I\right\}.$$  We show in Theorem~\ref{HRS theorem} that by choosing $I$ appropriately we obtain bad Noetherian stable domains birationally dominating $A$.

\index{Heinzer, W.}
Our theorem depends on the following lemma, due to Heinzer, Rotthaus\index{Rotthaus, C.} and Sally.  \index{Sally, J.~D.}
The result which we state  here is
weaker than what appears in \cite{HRS}, except that the last assertion of the lemma is stated there as
 $\widehat{A}/I \cong \widehat{R}$.  However, the  stronger form that we have used asserting the surjection of the canonical mapping  is justified by the proof in \cite{HRS}.

\index{Heinzer, W.} \index{Rotthaus, C.} \index{Sally, J.~D.}
\begin{lem} \label{HRS lemma} {\em (Heinzer-Rotthaus-Sally \cite[Corollary 1.27]{HRS})}
Let $A$ be a local Noetherian domain with  quotient field $F$, and
let $I$ be an ideal of $\widehat{A}$ with the property that each
associated prime $P$ of $I$ satisfies $P \cap A = 0$.
If the Krull dimension of $\widehat{A}/I$ is $1$,
then $R := F \cap (\widehat{A}/I)$ is a Noetherian domain of Krull
dimension $1$ and the canonical mapping $\phi:\widehat{A} \rightarrow \widehat{R}$ is a surjection having kernel $I$.   
\end{lem}

The following  lemma is an application of a standard fact about finite generation of modules over complete local rings.

\index{tightly dominates}
\begin{lem} \label{finite Noetherian} Let $A$ be a local Noetherian domain with maximal ideal ${\ff m}$.    If $R$ is a quasilocal  domain of Krull dimension $1$ that finitely dominates $A$, then $R$ is a Noetherian domain. If also $R$ tightly dominates $A$, then
 the canonical homomorphism  $\widehat{A} \rightarrow \widehat{R}$ is a surjection, where $\widehat{R}$ is the completion of $\widehat{R}$ in the $\ff m$-adic topology.
\end{lem}

\begin{proof}  Note  that ${\ff m}\widehat{A} \cdot \widehat{R} = {\ff m}\widehat{R}$, and that 
 $\widehat{R}$ is separated in the ${\ff m}\widehat{A}$-adic topology.
Now  since $R/{\ff m}R \cong \widehat{R}/{\ff m}\widehat{R}$ and $R/{\ff m}R$ is finite over $A$, there exist elements $\omega_1,\ldots, \omega_n \in \widehat{R}$ such that $ \widehat{R}/{\ff m}\widehat{R}$ is generated as an $\widehat{A}$-module by the images of these elements.    Thus since   $\widehat{A}$ is complete in the ${\ff m}\widehat{A}$-adic topology and the $\widehat{A}$-module $\widehat{R}$
 is complete and separated in the ${\ff m}\widehat{A}$-adic topology, $\widehat{R}$ is generated as an $\widehat{A}$-module by $\omega_1,\ldots,\omega_n$ \cite[Theorem 8.4, p.~58]{Ma}.  Hence since $\widehat{A}$ is a Noetherian ring, so is the finite $\widehat{A}$-module $\widehat{R}$.
  In particular, if $M$ denotes the maximal ideal of $R$, then  $M\widehat{R}$ is a finitely generated ideal of $\widehat{R}$.  Let $0 \ne r \in M$. Then since $R$ has Krull dimension $1$, there exists $i>0$ such that ${\ff m}^i \subseteq rR$.  Therefore, since $\widehat{R}/{\ff m}^i\widehat{R}$ is isomorphic as an $R$-algebra to $R/{\ff m}^iR$, it follows that $\widehat{R}/r\widehat{R}$ and $R/rR$ are isomorphic as $R$-algebras.  Consequently, since the $R$-module $M\widehat{R}/r\widehat{R}$ is finitely generated, so is the $R$-module $M/rR$.  Therefore,  
  $M$ is a finitely generated ideal of $R$, and  since $R$ is quasilocal of Krull dimension $1$, $R$ is a Noetherian domain.

 Finally, suppose that $R = A + {\ff m}R$, and let $\phi:\widehat{A} \rightarrow \widehat{R}$ denote the canonical map.  We claim
that   $\widehat{R}= \phi(\widehat{A}) + {\ff m}\widehat{R}$.  For suppose that $\l r_i + {\ff m}^iR\r \in \widehat{R}$. (We are viewing $\widehat{R}$ here as a subring of $\prod_i R/{\ff m}^iR$.) Then for each $i>0$, there exists ${m}_i \in {\ff m}R$ such that $r_i = r_1 + m_i$.  Since $R = A + {\ff m}R$, there exist $b \in A$ and $m \in {\ff m}R$ such that $r_1 = b +m$.  Therefore, $$\l r_i + {\ff m}^iR \r = \l r_1 + m_i + {\ff m}^iR\r = \l b + (m+m_i) + {\ff m}^iR \r \in \phi(\widehat{A}) + {\ff m}\widehat{R},$$ and hence $\widehat{R} = \phi(\widehat{A}) + {\ff m}\widehat{R}$.
 So
  if we reconsider the  elements $\omega_1,\ldots, \omega_n$ whose images generate $\widehat{R}/{\ff m}\widehat{R}$ as an $\widehat{A}$-module, we see that we may assume that $n =1$ and $\omega_1 = 1$.  Thus, appealing again to Theorem 8.4 of  \cite{Ma}, we conclude that $\widehat{R}$ is generated as an $\widehat{A}$-module by $\omega_1 = 1$. Consequently, $\widehat{R} =\phi( \widehat{A})$, so that $\phi$ is a surjection.
\end{proof}

We apply the lemmas in the next theorem, which asserts under certain conditions (to be clarified later in Corollary~\ref{generic 2} and Theorem~\ref{generic}), the existence of bad stable rings tightly dominating a  
given local Noetherian domain.  The theorem gives also a representation of all such stable rings.

\begin{thm} \label{HRS theorem} \label{pre excellent} Let   $A$ be a local Noetherian domain, not a DVR, that  is tightly dominated by a DVR $V$.     Let $F$ denote the quotient field of $A$, and let
 $P$ be the kernel
of the canonical homomorphism $\widehat{A} \rightarrow \widehat{V}$. Then $P \cap A = 0$ and 
  the following statements are equivalent for any ring $R$ properly between $A$ and $V$. 
 \begin{itemize} 
 \item[(1)]  $R$  is a
 bad  stable ring (necessarily Noetherian, by Lemma~\ref{finite Noetherian})  tightly dominating $A$.
 
 \item[(2)]   $R = F \cap (\widehat{A}/J)$ for some  $P$-primary  ideal $J$ of $\widehat{A}$ containing $P^2$.

\end{itemize}
Moreover, every bad stable domain between $A$ and $V$ tightly dominating $A$ contains the bad Noetherian stable domain  $R = F \cap (\widehat{A}/P^{(2)})$, where $P^{(2)}$ is the second symbolic power of $P$.   
\end{thm}

\begin{proof} It is clear that $P \cap A = 0$, since $V$ is a DVR dominating $A$.  
First suppose that $R$ is a bad stable ring properly between $A$ and $V$ that tightly dominates $A$.  Then $R$ has maximal ideal ${\ff m}R$, where ${\ff m}$ is the maximal ideal of $A$.  
%
%
%
Since $R$ tightly dominates $A$, we have  by Lemma~\ref{finite Noetherian} that  $R$ is a Noetherian domain and the canonical mapping $\widehat{A} \rightarrow \widehat{R}$ is a surjection.  Let $J$ denote the kernel of this mapping, so that $\widehat{R} \cong \widehat{A}/J$ and  $J \subseteq P$.  Let $Q$ be the prime ideal of $\widehat{R}$ corresponding to $P/J$.
By Lemma~\ref{finite Noetherian}, the canonical mapping $\widehat{A} \rightarrow \widehat{V}$ is a surjection, so that $\widehat{A}/P \cong \widehat{V}$, and hence $\widehat{A}/P \cong \widehat{R}/Q$ is a DVR.
Now since $R$ has Krull dimension $1$, the ${\ff m}$-adic and ideal topologies on $R$ are the same, so  $\widehat{R}$ can be viewed as the completion of $R$ in the ideal topology.  Thus by Theorem~\ref{complete char},
there is a prime ideal $Q'$ of $\widehat{R}$ such that $(Q')^2 = 0$ and $\widehat{R}/Q'$ is a DVR. Since then $F \cap \widehat{R}/Q'$ is another DVR containing $R$ and the normalization of $R$ is the DVR  $V$, this forces $V = F \cap \widehat{R}/Q'$. Consequently, by Lemma~\ref{HRS lemma}, 
$Q = Q'$, and hence $Q^2 = 0$, which implies   $P^2 \subseteq J \subseteq P$.  

To see next that $J$ is $P$-primary, note that 
 since the ${\ff m}$-adic and ideal topologies agree on $R$,  then
  $\widehat{R} \cong \widehat{A}/J$ is a torsion-free $A$-module \cite[Theorem 2.1, p.~11]{M1}.  Thus for all $0 \ne b \in A$, $(J:_{\widehat{A}}b) = J$.
We claim that the latter property  implies that   
  $J$ is $P$-primary. Indeed, since $P^2 \subseteq J$ and $\widehat{A}/J$ has Krull dimension $1$, the only prime ideals of $\widehat{A}$ containing $J$ are $P$ and the maximal ideal ${\ff m}\widehat{A}$ of $\widehat{A}$.   Thus to show that $J$ is $P$-primary it suffices to observe that
${\ff m}\widehat{A}$ is not an associated prime of $J$.  For if $x \in \widehat{A}$ with ${\ff m}\widehat{A} = (J:_{\widehat{A}} x)$, then choosing $0 \ne a \in {\ff m}$, we have $x \in (J:_{\widehat{A}}a) = J$, contrary to the choice of $x$.  
Therefore,  
the only associated prime of $J$ is $P$, and hence $J$ is $P$-primary.  
  Finally,   
 to finish the verification of (2), we need only observe that $R = F \cap (\widehat{A}/J)$.
Since $J$ is $P$-primary and $P \cap A = 0$, then as discussed at the beginning of the section, $F$ can be identified with a subring of the total quotient ring of $\widehat{A}/J$. Moreover,  since $\widehat{R} \cong \widehat{A}/J$, then  $\widehat{A}/J$ is faithfully flat over $R$, from which (2) now follows.   
%
%
%

Conversely, suppose  that $J$ is a $P$-primary ideal of $\widehat{A}$ such that $P^2 \subseteq J$ and $R = F \cap (\widehat{A}/J)$.   We prove that $R $ is a bad Noetherian stable domain tightly dominating $A$.
Since $P \cap A = 0$, we may appeal to Lemma~\ref{HRS lemma} to obtain that  the ring $R=F \cap (\widehat{A}/J)$ is a Noetherian domain of Krull dimension $1$ and the canonical $A$-algebra homomorphism $\widehat{A} \rightarrow \widehat{R}$ is surjective with kernel $I$.  Thus since ${\ff m}\widehat{A}$ is  the maximal ideal of $\widehat{A}$, the maximal ideal of  $\widehat{R}$ is ${\ff m}\widehat{R}$.
  Also, since $R$ has Krull dimension $1$ and ${\ff m}$ is a finitely generated ideal of $A$, the ${\ff m}$-adic and ideal topologies agree on $R$, so by Theorem~\ref{complete char}, $R$ is a bad stable domain, since $\widehat{A}/P$ is a DVR and $(P/J)^2 = 0$.
To see that ${\ff m}R$ is the maximal ideal of $R$, we   view $R$ as a subring of  $\widehat{R}$.  
 Then since ${\ff m}\widehat{R}$ is the maximal ideal of $\widehat{R}$ and $\widehat{R}/{\ff m}\widehat{R}$ and $R/{\ff m}R$ are isomorphic as  rings, it follows that 
 ${\ff m}R$ is the maximal ideal of $R$.  In summary, $R = F \cap (\widehat{A}/J)$ is a bad Noetherian stable domain with maximal ideal ${\ff m}R$.  Thus since $V = A + {\ff m}V$ and ${\ff m}R$ is the maximal ideal of $R$, it follows that $R = A + {\ff m}R$, so that $R$ tightly dominates $A$. Thus (2) implies (1). 
 
Next we observe  that the ring $R =F \cap (\widehat{A}/P^{(2)})$ is contained in each
bad stable ring between $A$ and $V$ that tightly
dominates $A$.  For let $R'$ be such a ring.  Then by the equivalence of (1) and (2), there exists a $P$-primary ideal $J$ of $\widehat{A}$ such that
 $P^2 \subseteq J$ and $R' = F \cap (\widehat{A}/J).$ Since $P^{(2)} \subseteq J$, it follows that $R \subseteq R'$. 
 

Finally we show that the embedding dimension of $R=F \cap (\widehat{A}/P^{(2)})$ is $1$ more than the embedding dimension of $\widehat{A}_P$.  Since $P^{(2)}$ is a $P$-primary ideal, it follows that $P/P^{(2)}$ is a torsion-free $\widehat{A}/P$-module.
   Therefore, $P/P^{(2)}$ is a   
finitely generated free module over the DVR $\widehat{A}/P$.  The rank of this free module is the dimension of the $(\widehat{A}_P/P_P)$-vector space $P_P/P^{(2)}_P = P_P/P^2_P$, and the dimension of this vector space is in turn the embedding dimension, $n$, of $\widehat{A}_P$.  Thus the ideal $P/P^{(2)}$ of $\widehat{A}/P^{(2)}$ is minimally generated by $n$ elements.  Passing now from $\widehat{A}/P^{(2)}$  to the isomorphic ring 
 $\widehat{R}$ (Lemma~\ref{HRS lemma}), there exists a prime $Q$ of $\widehat{R}$ such that $Q^2 = 0$, $Q$ is minimally generated by $n$ elements, and ${\widehat{R}}/Q$ is a DVR.  Let $t$ be in the maximal ideal $M$ of $\widehat{R}$ such that $M = t{\widehat{R}} + Q$.  We claim that $(t\widehat{R} + M^2) \cap (Q + M^2) = M^2$.  Let $x \in \widehat{R}$, $m_1,m_2 \in M^2$ and $ q \in Q$ such that $xt + m_1 = q+m_2$.  If $x \in M$, then $xt \in M^2$, so that $xt+m_1 \in M^2$ as claimed.  Otherwise, $x$ is a unit in $\widehat{R}$, and hence $t \in Q+M^2$.  But then $M = t\widehat{R} + Q \subseteq Q + M^2$, which is impossible since $\widehat{R}/Q$ is a DVR with maximal ideal $M/Q$.  Therefore, $(t\widehat{R} + M^2) \cap (Q + M^2) = M^2$, and hence
  $$M/M^2 = (t\widehat{R}+{M}^2)/{M}^2 \oplus (Q+{M}^2)/{M}^2.$$  Moreover, by Lemma~\ref{stable null}, there exists $x \in {M}$ such that $M^2 = xM$.  Therefore, $(Q + {M}^2)/{M}^2 \cong Q/(Q \cap xM) = Q/xQ$.  However, $xQ \subseteq {M}Q \subseteq {M}^2 \cap Q = x{M} \cap Q = xQ$, so $xQ = {M}Q$, and hence 
 $(Q + {M}^2)/{M}^2 \cong Q/MQ$.  Thus  
  the dimension of the $\widehat{R}/M$-vector space $(Q + {M}^2)/{M}^2$ is $n$, which using Nakyama's Lemma proves that $M$ is minimally generated by $n+1$ elements.  Since the embedding dimension of $R$ and $\widehat{R}$ agree, the proof is complete.     
\end{proof}

\begin{rem} {\em Theorem~\ref{HRS theorem} does not account for all of the bad stable domains between $A$ and $V$.  For example, if we let $R$ be as in (2), then the ring $R_1=\End({\ff m}R)$ is a stable domain between $R$ and $V$ (Theorem~\ref{convex}), but its maximal ideal is not extended from ${\ff m}$, a necessary condition in order for the ring to fall into the classification in the theorem.  Indeed, ${\ff m}R_1 = {\ff m}R$, so that if ${\ff m}R_1$ is the maximal ideal of $R_1$, then since it is stable, it is necessarily principal in $R_1$.  But then $R_1$ is a DVR, and hence $R_1 = V$, which forces $V$ to be a fractional ideal of $R$.  This is impossible since $V/R$ is a divisible $R$-module (divisibility follows from the fact that $V$ is a DVR with $V = R +{\ff m}V$). Therefore, the ring $R_1$ cannot be described as in (1) of the theorem.  This argument shows that in general a bad stable domain that tightly dominates a quasilocal domain $A$ will have stable overrings that do not tightly dominate $A$.
}
\end{rem}


\begin{cor} \label{generic 2}  The following statements are equivalent for a local Noetherian
domain $A$ that is not a DVR.

\begin{itemize}

\item[{(1)}] There is a prime ideal $P$ of $\widehat{A}$ such
that $P \cap A = 0$ and $\widehat{A}/P$ is a DVR.

\index{DVR!tightly dominating}
\item[{(2)}]  $A$ is tightly dominated by a DVR.

\index{stable domain!tightly dominating}
\item[{(3)}]  $A$ is tightly dominated by a bad Noetherian stable
domain.

\index{$2$-generator ring!tightly dominating}
\item[(4)]  $A$ is tightly dominated by
a bad $2$-generator  domain.

\end{itemize}

\end{cor}

\begin{proof}
(1) $\Rightarrow$ (2)  Let $V = F \cap (\widehat{A}/P)$, where $F$ denotes the quotient field of $A$.   Then since $\widehat{A}/P$ is a DVR with maximal ideal generated by the image of the maximal ideal ${\ff m}$ of $A$, and since the  residue field of $\widehat{A}/P$ is  $A/{\ff m}$, it follows that $V$ is a DVR with maximal ideal ${\ff m}V$ and residue field $A/{\ff m}$.  Therefore,   $V$ tightly dominates $A$.

(2) $\Rightarrow$ (3) By Theorem~\ref{HRS theorem}, $R = F \cap  (\widehat{A}/P^{(2)})$ is  a bad Noetherian stable domain that tightly dominates $A$.

(3) $\Rightarrow$ (4) Let $R$ be a bad Noetherian stable domain that tightly dominates $A$.   Apply Theorem~\ref{convex} to the ring $R$ in (3) to obtain a bad stable domain $T$ of embedding dimension $2$ such that $T/R$ is a divisible $R$-module.  Since the multiplicity and embedding dimension of a local stable ring agree, it follows that $T$ is a $2$-generator ring.
  Moreover,  the divisibility of $T/R$ implies that $T = R +{\ff m}T$.  Thus since $R = A + {\ff m}R$, we see that $T = A + {\ff m}T$, and hence $T$ tightly dominates $A$.

(4) $\Rightarrow$ (2) Let $R$ be a bad $2$-generator domain that tightly dominates $A$.  By Proposition~\ref{ar stable char}, $\overline{R}$ is a DVR and $\overline{R}/R$ is a divisible $R$-module.  Hence $\overline{R} = R + {\ff m}\overline{R} = A + {\ff m}\overline{R}$, and $\overline{R}$ tightly dominates $A$.

(2) $\Rightarrow$ (1)  Suppose $V$ is a DVR that tightly dominates $A$.    Since $V = A + {\ff m}V$, we have by Lemma~\ref{finite Noetherian} that the canonical mapping $\phi:\widehat{A} \rightarrow \widehat{V}$ is surjective.  Since $V$ is a DVR, so is $\widehat{V}$, and hence the kernel $P$ of this mapping is a prime ideal of $\widehat{A}$.  If $b \in P \cap A$, then  $b \in {\ff m}^iV$ for all $i>0$, which since $V$ is a DVR forces $b = 0$.  Therefore, $P \cap A = 0$.
\end{proof}

From Theorem~\ref{HRS theorem} we deduce a bound on the embedding dimension of bad stable rings dominating an excellent local Noetherian domain:

\begin{cor} \label{excellent} Let $A$ be an excellent local Noetherian domain.  If $A$ has   dimension $d>1$ and the generic formal fiber of $A$ has dimension $d-1$, then there exists a bad Noetherian stable ring finitely dominating $A$ and having embedding dimension $d$, and every bad stable ring finitely dominating $A$  has embedding dimension at most $d$.     
\end{cor}

\begin{proof}  Since the generic formal fiber has dimension $d-1$, there exists by a theorem of Heinzer, Rotthaus and Sally, a DVR $V$ finitely dominating $A$; cf.~Theorem~\ref{generic}.  Since $V$ finitely dominates $A$, there exist $x_1,\ldots,x_n \in V$  such that $V$ tightly dominates $B:=A[x_1,\ldots,x_n]_{\bf n}$, where  ${\ff n}$ is the contraction of the maximal ideal of $V$ to $A[x_1,\ldots,x_n]$.  Since the residue field of $V$ is finite over that of $A$, the residue field of $B$ is finite also over the residue field of $A$.  Thus,  since $A$ is universally catenary,  the Dimension Formula implies  the dimension of $B$ is $d$ \cite[Theorem 15.6, p.~119]{Ma}.  Also,   since $A$ is excellent and $B$ is a localization of a finitely generated $A$-algebra,  the generic formal fiber of $B$ is regular \cite[Theorem 77, p.~254]{Ma3}.  Therefore, for the prime ideal $P$ of $\widehat{B}$ corresponding to $V$ as in Theorem~\ref{pre excellent}, the ring $\widehat{B}_P$ is a regular local ring, and hence has embedding dimension $d-1$.  Applying Theorem~\ref{pre excellent}, there exists a bad Noetherian stable ring finitely dominating $A$ and having embedding dimension $d$.  
Now let $S$ be a bad stable ring finitely dominating $A$.  Then by Proposition~\ref{ar stable char}, $S$ is tightly dominated by a DVR $U$, and as above, $S$ tightly dominates  a local $d$-dimensional Noetherian domain $B$ essentially of finite type over $A$.  It follows that $U$ tightly dominates $B$, and by Theorem~\ref{pre excellent} and the above argument, there exists a bad stable ring of embedding dimension $d$  contained in  $S$.  Hence by Theorem~\ref{convex}, $S$ has embedding dimension at most $d$.  
\end{proof}

\begin{cor}
 Let $k$ be a field, and let $A$ be an affine $k$-domain with Krull dimension $d>1$ and quotient field $F$. Then there exists a bad stable ring between $A$ and $F$ having  embedding dimension $d$ and residue field  finite over $k$.  Moreover, every bad stable ring between $A$ and $F$ having residue field finite over $k$ has embedding dimension at most $d$.    
 \end{cor}
 
 \begin{proof} 
 Matsumura has shown that the generic formal fiber of a $d$-dimensional domain  essentially of finite type over a field has dimension $d-1$ \cite[Theorem 1]{Mat2}.  Thus by Corollary~\ref{excellent}, there exists a bad stable ring between $A$ and $F$ having embedding dimension $d$ and residue field finite over $k$.  If $R$ is  a bad stable ring between $A$ and $F$ having residue field finite over $k$, then $R$ finitely dominates $A_{M \cap A}$, where $M$ is the maximal ideal of $R$.  Since $A_{M \cap A}$ is excellent of dimension $d$, Corollary~\ref{excellent} implies that $R$ has embedding dimension at most $d$.  
 \end{proof}


Theorem~\ref{HRS theorem} shows  there exists a smallest bad  stable domain $R$ between $A$ and $V$, {\it as long as one restricts to those domains $R$ tightly dominating $A$.}  We remove this restriction of tight domination by $R$ and show in Theorem~\ref{minimal stable Noetherian} that  among all bad stable rings between $A$ and $V$,  $R = \Ker d_{V/A}$ is  the smallest.   (It follows easily from properties of derivations that their kernels are rings.)    

\index{DVR!tightly dominating}
\begin{lem} \label{minimal stable} Let $A$ be a quasilocal domain tightly dominated by a DVR  $V$, and let   $R=\Ker d_{V/A}$.
Then
  $V/R$ and 
$\Omega_{V/A}$ are isomorphic as $R$-modules, and if $R \subsetneq V$, then $R$ is a bad stable ring contained in every bad stable ring between $A$ and $V$.  
\index{stable domain!tightly dominating}


 \end{lem}

 \begin{proof}
   Let ${\ff m}$ denote the maximal ideal of $A$.   Since $V$ tightly dominates $A$, we have  $V = A + {\ff m}^iV$ for all $i >0$.      
   This  implies  that $V/A$ is a divisible $A$-module.
  For if $0 \not = a \in {\ff m}$, then since $aV$ and ${\ff m}V$ are primary with respect to the maximal ideal of $V$, there exists $i>0$ such that ${\ff m}^iV \subseteq aV$, and hence $V = A + aV$, which implies $V/A$ is a divisible $A$-module.   
  Next,  
       an argument such as that in the proof of (4) implies (5) in Lemma~\ref{start} shows then  
 that
 $d_{V/A}(V) = \Omega_{V/A}$, and hence, since $d_{V/A}$ is an $R$-module homomorphism (indeed, since $R = \Ker d_{V/A}$,  then $d_{V/A}$ is $R$-linear), then
  $V/R \cong \Omega_{V/A}$ as $R$-modules.
 To complete the proof, we assume that $R \ne V$.   First we claim that $R$ is a bad stable domain.
  As a
 homomorphic image of $V/A$, $V/R$ is a torsion divisible
 $A$-module.  But since $\Omega_{V/A}$ is a $V$-module and  $V/R \cong \Omega_{V/A}$ as $R$-modules, it follows that the $R$-module $V/R$
 admits a $V$-module structure.  Also since $V/A$ is a divisible $A$-module and $A$ and $R$ share the same quotient field, then $V/R$ is a divisible $R$-module.  Therefore,
 by Lemma~\ref{start},
 $R \subseteq V$ is a quadratic extension.  Since $R \ne V$, we have
 by Proposition~\ref{ar stable char} that $R$ is a bad stable ring.  Applying  Proposition~\ref{ar stable char} again shows  every ring $R'$ with $R \subseteq R' \subsetneq V$ is a bad stable ring.

Now suppose that   $R'$ is a bad stable ring with $A
 \subseteq R' \subseteq V$.  Then by Proposition~\ref{ar stable char}, $R'$ has normalization a DVR, so necessarily $V$ is the integral closure of $R'$.
Also, by Proposition~\ref{ar stable char},  $R' \subseteq V$
 is a quadratic extension, so since $V/R'$ is a divisible $R'$-module,
 Lemma~\ref{start}(5) implies that the exterior differential  $d_{V/R'}$ maps onto
 $\Omega_{V/R'}$ with kernel $R'$.  However, since $d_{V/R'}$ is an $A$-linear derivation, there exists, as discussed in the Introduction, a $V$-module homomorphism
  $\alpha:\Omega_{V/A} \rightarrow \Omega_{V/R'}$ such that $d_{V/R'}
  = \alpha \circ d_{V/A}$.  Thus $R = \Ker d_{V/A} \subseteq \Ker
  d_{V/R'} = R'$, which completes the proof.
 \end{proof}


\index{DVR!tightly dominating} \index{stable domain!existence of}
\begin{thm} \label{minimal stable Noetherian} Let $A$ be a local Noetherian domain, not a DVR, that is tightly dominated by a DVR $V$.  Then the ring  $R = \Ker d_{V/A}$ is a bad stable domain contained in every bad stable ring between $A$ and $V$.  
 \end{thm}

\begin{proof}  In light of Lemma~\ref{minimal stable} all that needs to be shown is that $R$ is a proper subring of $V$.  In fact, since $R$ is a subring of every bad stable domain between $A$ and $V$, we need only show there exists a bad stable domain between $A$ and $V$.  The existence of such a ring is guaranteed by Theorem~\ref{HRS theorem}, so the proof is complete.
\end{proof}


We use the preceding ideas next to characterize  the local Noetherian domains which are finitely dominated by a bad stable ring.  
Matlis has proved that if $A$ is an analytically ramified local
Noetherian domain of Krull dimension $1$, then there exists a bad
2-generator ring $R$ between $A$ and its quotient field \cite[Theorem 14.16]{M1}.  
In Theorem~\ref{Matlis ar corollary} we recover a version of Matlis' theorem from a different point of view.
%
%

\begin{thm} \label{Matlis ar corollary} {\em (Matlis)}  Every one-dimensional analytically ramified local Noetherian domain is  finitely dominated by a  bad $2$-generator  ring.
\end{thm}

\begin{proof}
We first show that there exists an  analytically ramified local Noetherian ring finitely dominating $A$  whose normalization is a DVR.  
Since $\overline{A}$ is a Dedekind domain it has only finitely many maximal ideals.  For each of these maximal ideals, choose an element in it but in no other maximal ideal, and let $T$ be the $A$-algebra generated by these finitely many elements.  Then $T$ has the same number of maximal ideals as $\overline{A}$.  Since $\overline{A}$ is not a finite $T$-module (if it were, it would force $\overline{A}$ to be a finite $A$-module, contrary to assumption), there exists a maximal ideal
 $M$ of $T$ such that the ring
 $T_{M}$ is analytically ramified.  For otherwise, if each localization of $T$ at a maximal ideal $M$ is analytically unramified, then each ring $\overline{A}_{M}$ is a fractional ideal of $T_{M}$.  But this then implies
that there exists an element $t$ of $T$ such that $t \overline{A} \subseteq T$, a contradiction to the fact that $\overline{A}$ is not a finite module over the Noetherian ring $T$.  Thus  there is a maximal ideal $M$ of $T$ such that the ring  $A':={T}_{M}$ is an analytically ramified local Noetherian domain with integral closure the  DVR $V := \overline{A}_{N}$, where $N$ is the unique maximal ideal of $\overline{A}$ lying over $M$.
Finally, to see that $A'$ finitely dominates $A$, let ${\ff m}'$ be the maximal ideal of $A'$.   Since $A' = T_M$, we have $A' = T + {\ff m}'$, and since $A'$ has Krull dimension $1$ and its maximal ideal ${\ff m'}$ is finitely generated,  there exists $k>0$ such that $({\ff m}')^k
 \subseteq {\ff m}A'$.  Now since $A' = T  + {\ff m}'$ and $T$ is a finite $A$-module, then $A'/{\ff m'}$ is a finite $A$-module, and hence it follows that  $A'/({\ff m}')^k$ is a finite $A$-module.  Therefore, since $A'/{\ff m}A'$ is an $A$-homomorphic image of $A'/({\ff m}')^k$, we conclude that $A'/{\ff m}A'$ is a finite $A$-module, and this proves that $A'$ is an analytically ramified local Noetherian ring that finitely dominates $A$ and has normalization the DVR $V$.    

We note next that 
there exists a finite extension $A''$ of $A'$ such that $V$ tightly dominates $A''$.  For since $V$ finitely dominates $A'$, there exist $x_1,\ldots,x_n \in V$ such that $V = A'x_1 + \cdots + A'x_n + {\ff m}V$.  Since $V$ is the normalization of $A'$, the ring $A'':=A'[x_1,\ldots,x_n]$ is a finite extension of $A'$ having normalization $V$, and since $V$ is not a finite $A'$-module, it follows that $V$ is not a finite $A''$-module.  Thus $A''$ is a one-dimensional analytically ramified local Noetherian domain that is tightly dominated by $V$.   
We now apply either Theorem~\ref{HRS theorem} or~\ref{minimal stable Noetherian} to
 obtain a bad stable ring
$R$ with $A'' \subseteq R \subseteq V$.  
 Applying Theorem~\ref{convex}, we may in fact assume that $R$ has embedding dimension $2$,  and hence since in a stable ring, multiplicity and embedding dimension agree,  $R$ is a bad $2$-generator ring.
 Moreover, 
  $R$ and $V$ have the same residue field.  Indeed, since $V$ tightly dominates $A''$, $V = A'' + {\ff m}''V$, where ${\ff m}''$ is the maximal ideal of the one-dimensional local domain $A''$.  Thus since the maximal ideal $M$ of $R$ contains  ${\ff m}''$, it follows that $V = R + MV$, and hence $V/MV \cong R/M$.  Therefore, $R$ and $V$ have the same residue field.     Similarly, since $V = A'' + {\ff m}''V$,  then $V$ and $A''$ have the same residue field,   
  and therefore, $A''$ and $R$ have the same residue field.  Now $R/M$ is a cyclic $A''$-module, so since $A''$ is a finite $A'$-module,  $R/M$ is a finite $A'/{\ff m}A'$-module.   Thus, since $A'/{\ff m}A'$ is a finite $A$-module, we see that $R/M$ is a finite $A$-module.  Moreover, since $M$ is a finitely generated ideal of $R$ (necessarily $R$ is Noetherian, since it is an overring of the one-dimensional Noetherian domain $A$), there exists $k>0$ such that $M^k \subseteq {\ff m}R$.  It follows that $R/M^k$ is a finite $A$-module, so since $R/{\ff m}R$ is an $A$-homomorphic image of $R/M^k$, we conclude that $R/{\ff m}R$ is a finite $A$-module.
\end{proof}


Thus 
a one-dimensional local Noetherian domain $A$ is finitely dominated by a bad stable  ring if and only if $A$ is  
 analytically ramified.   In higher dimensions, we have
%


\begin{cor} \label{generic lemma} \label{generic} Let $A$ be a local Noetherian
domain  with  dimension $d>1$.  Then the following statements  are equivalent.

\begin{itemize} 

\item[{(1)}]  $A$ is finitely dominated by
a bad Noetherian stable ring.  

\item[(2)]  $A$ is finitely dominated by
an analytically ramified one-dimensional local Noetherian ring.

\item[(3)]  $A$ is tightly dominated by
an  analytically ramified  one-dimensional local Noetherian ring.

\item[{(4)}]  $A$ is finitely dominated by a DVR.


\item[{(5)}] The dimension of the generic formal fiber of $A$ is $d-1$.

\index{essentially of finite type}
\end{itemize} If also $A$ is excellent, then the 
stable ring in (1) can be chosen to have embedding dimension $d$ but no bigger.  
Moreover, these five equivalent conditions   are satisfied when  $A$ is essentially of finite type over a field and has dimension $d>1$.
\end{cor}

\begin{proof}
(1) $\Rightarrow$ (2) This is clear.  
  
(2) $\Rightarrow$ (4)  If $R$ is a one-dimensional analytically ramified  local Noetherian ring that finitely dominates $A$, then since 
 by Theorem~\ref{Matlis ar corollary}, $R$ is finitely dominated by a
 bad stable ring $S$.  The ring $S$ in turn is tightly dominated by a DVR $V$ (Proposition~\ref{ar stable char}), and hence $V$ finitely dominates $A$.



(4) $\Rightarrow$ (5)
This is due to Heinzer, Rotthaus and Sally 
\cite[Corollary 2.4]{HRS}.


(5) $\Rightarrow$ (3)
 By (5), there exists a height $d-1$ prime ideal $P$  of $\widehat{A}$ such that $\widehat{A}/P$ has Krull dimension $1$ and $A \cap P = 0$.  Since $d>1$, $P$ is not a minimal prime ideal of $\widehat{A}$.  With the aim of applying Lemma~\ref{HRS lemma}, observe that  $P^{(2)} \subsetneq P$, since otherwise  $P^2\widehat{A}_P = P\widehat{A}_P$, which, since $P$ is finitely generated, implies that $P\widehat{A}_P = 0$, and hence that $P$ is a minimal prime ideal of $\widehat{A}$, a contradiction.  Therefore, 
 we may apply  Lemma~\ref{HRS lemma} to obtain that $R:=(\widehat{A}/P^{(2)}) \cap F$ is a Noetherian domain of Krull dimension $1$ with $\widehat{R} \cong \widehat{A}/P^{(2)}$.  Also, $R$ is analytically ramified because $P \ne P^{(2)}$, so that $\widehat{R}$ contains nilpotents.
      Since $\widehat{A}/P^{(2)}$ has a unique minimal prime ideal, namely, $P/P^{(2)}$, it follows that the normalization of $R$ is a DVR \cite[Theorem 10.5]{M1}.  Thus we have constructed a one-dimensional analytically ramified local Noetherian domain $R$ that birationally dominates $A$.  Moreover, since $\widehat{R} \cong \widehat{A}/P^{(2)}$, it follows that $\widehat{R}$ and $\widehat{A}$, and hence $R$ and $A$, have the same residue field and that ${\ff m}\widehat{R}$ is the maximal ideal of $\widehat{R}$.  Consequently, ${\ff m}R$ is the maximal ideal of $R$, and hence since $R$ and $A$ have the same residue field, $R = A + {\ff m}R$.   Therefore, $R$  tightly dominates $A$.


(3) $\Rightarrow$ (1)    Let $R$ be  an analytically ramified 
 one-dimensional local Noetherian ring that finitely dominates $A$.  Then by Theorem~\ref{Matlis ar corollary}, $R$ is finitely dominated by a bad Noetherian stable ring, and hence this ring finitely dominates $A$.

If also $A$ is excellent and (1) -- (5) hold, then by Corollary~\ref{excellent}, $A$ is finitely dominated by a bad stable ring of embedding dimension $d$ and every bad stable ring finitely dominating $A$ has embedding dimension at most $d$.  
Finally, the last assertion of the theorem, that all five conditions are equivalent when $A$ is essentially of finite type over a field, follows from
\cite[Theorem 2]{Mat2}, where it is shown that for such a ring the dimension of the generic formal formal fiber is $d-1$.
\end{proof}

\begin{rem} \label{generic remark} {\em Statements (4) and (5) are equivalent also when $A$ has Krull dimension $1$; see \cite[Corollary 2.4]{HRS}.}
\end{rem}  

{\bf Acknowledgment.}  I thank the referee for helpful comments that improved the presentation of the paper.


\begin{thebibliography}{FHP34}



\bibitem{Akizuki} Y.~Akizuki, Einige Bemerkungen \"uber prim\"are Integrit\"atsbereiche mit Teilerkettensatz, Proc. Phys.--Math. Soc. Japan {17} (1935), 327--336.







\bibitem{Bass} H.~Bass, On the ubiquity of Gorenstein rings, Math Z.
82 (1963), 8--28.










\bibitem{DK} J.~A.~Drozd and V.~V.~Kiri\v{c}enko, On quasi-Bass
orders, Math. USSR Izv. 6 (1972), 323-365.


\bibitem{ElG} S.~El Baghdadi and S.~Gabelli,  Ring-theoretic properties of P$v$MDs, Comm. Algebra 35 (2007), no. 5, 1607Ð-1625. 














\bibitem{FS} L.~Fuchs and L.~Salce, {\it Modules over non-Noetherian
domains}, Math. Surveys 84, 2001.


\bibitem{GP} S.~Gabelli and G.~Picozza,  Star stable domains, J. Pure Appl. Algebra 208 (2007), no. 3, 853-Ð866.







\bibitem{HLS} W.~Heinzer, D.~Lantz and K.~Shah, The Ratliff-Rush
ideals in a Noetherian ring, Comm. Alg. 20 (1992), 491--522.

\bibitem{HRS} W. ~Heinzer, C.~Rotthaus and J.~Sally,
 Formal fibers
and birational extensions,  Nagoya Math. J.  131  (1993), 1--38.

\bibitem{KM} S.~Kabbaj and A.~Mimouni, $t$-class semigroups of integral domains, J. Reine Angew. Math. 612 (2007), 213Ð-229.







\bibitem{Krull} W.~Krull, Dimensionstheorie in Stellenringen, J.~Reine Angew.~Math. 179 (1938), 204--226.









\bibitem{Lipman} J.~Lipman, Stable ideals and Arf rings,  Amer. J. Math.  {\bf 93}  (1971) 649--685.




\bibitem{MatlisNoetherian}  E.~Matlis, Injective modules over Noetherian rings,  Pacific J. Math.  8  (1958) 511--528.

\bibitem{M3} E.~Matlis, Injective modules over Pr\"ufer rings,  Nagoya Math. J  15  (1959)
57--69.

\bibitem{M1} E.~Matlis, {\it 1-dimensional Cohen-Macaulay rings}, Lecture
Notes in Math. {327}, Springer-Verlag, 1973.




\bibitem{M2} E.~Matlis, {\it Torsion-free modules}, The University of Chicago
Press, 1972.


\bibitem{Ma3} H.~Matsumura, {\it Commutative algebra}, Benjamin/Cummings Publishing Company, 1980.  

\bibitem{Ma}  H.~Matsumura, {\it Commutative ring theory}, Cambridge
University Press, 1986.

\bibitem{Mat2} H.~Matsumura, On the dimension of formal fibres of a
local ring, in: {\it Algebraic geometry and commutative algebra in
honor of Masayoshi Nagata},  Vol. I,  261--266, Kinokuniya, Tokyo,
1988.

\bibitem{Mim} A.~Mimouni,  Ratliff-Rush closure of ideals in integral domains, Glasg. Math. J. 51 (2009), no. 3, 681Ð689. 

\bibitem{Na}  M.~Nagata, {\it Local rings}.
Interscience Tracts in Pure and Applied Mathematics, No. 13, John
Wiley \& Sons, New York-London, 1962.

\bibitem{Naga} R.~Nagasawa, Some remarks on one-dimensional 
local domains,  
Publ. RIMS, Kyoto Univ. 
11 (1975), 21--30. 
 



\bibitem{OlStructure} B.~Olberding, On the structure of stable
domains, Comm. Algebra {30} (2002), no. 2, 877--895.

\bibitem{OlClass} B.~Olberding, On the classification of stable
domains, J. Algebra 243 (2001), 177--197.


\bibitem{OlRend} B.~Olberding, Stability, duality, 2-generated
ideals and a canonical decompostion of modules, Rend. Sem. Mat.
Univ. Padova 106 (2001), 261--290.

\bibitem{OlSurvey} B.~Olberding,
Stability of ideals and its applications, in {\it Ideal theoretic
methods in commutative algebra (Columbia, MO, 1999)},  319--341,
Lecture Notes in Pure and Appl. Math., 220, Dekker, New York, 2001.


\bibitem{OlbAR} B.~Olberding, One-dimensional bad Noetherian rings, submitted.

\bibitem{OlbCounter} B.~Olberding, A counterpart to Nagata idealization, J. Algebra 365 (2012), 199Ð-221.

\bibitem{PT} G.~Picozza and F.~Tartarone,  Flat ideals and stability in integral domains, J. Algebra 324 (2010), no. 8, 1790Ð-1802. 












\bibitem{SV} J.~Sally and W.~Vasconcelos, Stable rings, J. Pure
Appl. Algebra 4 (1974), 319-336.

\bibitem{SVBull} J.~Sally and W.~Vasconcelos, Stable rings and a problem of Bass,
Bull. Amer. Math. Soc. 79 (1973), 575--576.












\bibitem{Sega} L.~Sega,  Ideal class semigroups of overrings, J. Algebra 311 (2007), no. 2, 702-Ð713. 


\bibitem{Schmidt} F.~K.~Schmidt, \"Uber die Erhaltung der Kettens\"atze der Idealtheorie bei beliebigen endlichen 
K\"orpererweiterungen, Math. Zeit. 41  (1936), 443--450. 

\bibitem{Zan2} P.~Zanardo,  The class semigroup of local one-dimensional domains, 
J. Pure Appl. Algebra 212 (2008), no. 10, 2259-Ð2270. 


\bibitem{Zan} P.~Zanardo,  Algebraic entropy of endomorphisms over local one-dimensional domains, J. Algebra Appl. 8 (2009), no. 6, 759Ð-777.


 
 
\end{thebibliography}
\end{document}